\documentclass[english]{article}
\usepackage[T1]{fontenc}
\usepackage[latin9]{inputenc}
\usepackage{amsthm}
\usepackage{amsmath}
\usepackage{graphicx}
\usepackage{setspace}
\usepackage{amssymb}
\makeatletter

\providecommand*{\perispomeni}{\char126}
\AtBeginDocument{\DeclareRobustCommand{\greektext}{%
  \fontencoding{LGR}\selectfont\def\encodingdefault{LGR}%
  \renewcommand{\~}{\perispomeni}%
}}
\DeclareRobustCommand{\textgreek}[1]{\leavevmode{\greektext #1}}
\DeclareFontEncoding{LGR}{}{}

\newcommand{\lyxmathsym}[1]{\ifmmode\begingroup\def\b@ld{bold}
  \text{\ifx\math@version\b@ld\bfseries\fi#1}\endgroup\else#1\fi}

\theoremstyle{plain}
\newtheorem{thm}{Theorem}[section]
  \theoremstyle{plain}
  \newtheorem{lem}[thm]{Lemma}
  \theoremstyle{definition}
  \newtheorem{defn}[thm]{Definition}
  \theoremstyle{remark}
  \newtheorem{rem}[thm]{Remark}
  \theoremstyle{plain}
  \newtheorem{prop}[thm]{Proposition}
  \theoremstyle{plain}
  \newtheorem{cor}[thm]{Corollary}

\makeatother

\usepackage{babel}

\begin{document}

\title{Complete intersection quiver settings with one dimensional vertices}

\author{Dániel Joó\\
 {\small Central European University}\\
 {\small e-mail: joo\_daniel@ceu-budapest.edu}}
\maketitle
\begin{abstract}
We describe the class of quiver settings with one dimensional vertices
whose semi-simple representations are parametrized by a complete intersection
variety. We show that these quivers can be reduced to a one vertex
quiver with some combinatorial reduction steps. We also show that
this class consists of the quivers from which we can not obtain two
specific non complete intersection quivers via contracting strongly
connected components and deleting subquivers. We also prove that the
class of coregular quivers with arbitrary dimension vector that has
been described earlier via reduction steps, can be described by not
containing a specific subquiver in the above sense.
\end{abstract}

\section{Introduction}

Representations of a quiver with a fixed dimension vector (quiver
setting) are parametrized by a vector space together with a linear
action of a product of general linear groups such that two point belong
to the same orbit if and only if the corresponding representations
are isomorphic. Therefore the quotient constructions of algebraic
geometry can be applied to give an approximation to the problem of
classification of isomorphism classes of representations. The simplest
quotient varieties, the so called affine quotients are defined in
terms of invariant polynomial functions on the representation spaces.
Although these affine quotient varieties turn out to reflect faithfully
the class of semisimple representations only, the study of them is
motivated by more sophisticated quotient constructions as well. Geometric
invariant theory has been applied by A. King \cite{King} to construct
non-trivial projective quotients (even in cases when the affine quotient
is a single point). It was shown by Adriaenssens and Le Bruyn \cite{ArdriaenssensandLeBruyn}
that the study of the local structure (say singularities) of these
projective quotients (moduli spaces) can be reduced to the study of
affine quotients of other quiver settings. (See \cite{Domokos-Smooth}
for an application illustrating the power of this method.) Therefore
the results on the affine quotient varieties of representation spaces
of quivers have relevance also for the study of more general moduli
spaces of quiver representations.

The class of coregular quiver settings (whose corresponding quotient
space is a smooth variety) has been described in \cite{Coregular}
by Bocklandt via establishing some reduction steps that simplify the
structure of the quiver without changing the singularities of the
corresponding variety. The proof for the main result in this paper
contains an error, but we will show in section 4 that it can be easily
fixed. The same method has been applied in \cite{CI} to describe
quiver settings with complete intersection quotient varieties (C.I.
quiver settings for short) in the special case when the quiver is
symmetric and has no loops. The general case however seem to be very
difficult to understand from this approach. In this paper we will
be concerned with the special case when the values of the dimension
vector are all one (but there is no restriction on the structure of
the quiver). Though this is a strong restriction from the point of
view of representation theory, it still covers a rather interesting
class, since the corresponding affine quotients are toric varieties.
The question when a toric variety is a complete intersection received
considerable attention in the literature, see for example \cite{Fisher}
or \cite{Gitler-Reyes-Villareal}. 

Our main results are contained in section 6, where we will establish
a new reduction-step for quivers with one dimensional vertices and
show that all C.I. quiver settings on at least two vertices can be
reduced. We will also introduce the notion of descendant, which is
similar to graph-theoretic minors (the difference is that only strongly
connected components can be contracted) and describe the class in
question as the quivers that do not contain certain forbidden descendants.
In section 7 we will show that the coregular quiver settings (with
arbitrary dimension vectors) can also be described by not containing
a single forbidden descendant. These results give some hope that a
simlar statement can be formulated in the general case for C.I. quiver
settings as well. 

Throughout this paper we will work over an algebraically closed field
of characteristic zero, which will be denoted it by $\mathbb{C}$.
This is convenient since we will use several results of Le Bruyn and
Procesi \cite{LeB-Proc}, and Raf Bocklandt \cite{CI,Coregular},
who worked with this assumption. However as it was shown by Domokos
and Zubkov in \cite{Domokos-Zubkov:charp} many of the results extend
to fields with positive characteristic as well. For example the classification
of quivers with genuine simple representation we will recall below,
holds over an arbitrary field.\medskip{}

\textbf{Acknowledgment:} The author would like to thank his supervisor
Mátyás Domokos for his invaluable help and guidance in his work.

\section{Preliminaries}

A \textit{quiver} $Q=(V,\: A,\: s,\: t)$ is a quadruple consisting
of a set of vertices $V$, a set of arrows $A$, and two maps $s,\: t:\: A\rightarrow V$
which assign to each arrow its starting and terminating vertex (loops
and multiple arrows are possible). A \textit{representation} $X$
of $Q$ is given by a vector space $X_{v}$ for each $v\in V$, and
a linear map $X_{\rho}:X_{s(\rho)}\rightarrow X_{t(\rho)}$ for each
$\rho\in A$. The dimension vector $\alpha:V\rightarrow\mathbb{N}$
of a representation $X$ is defined by $\alpha(v)=dim(X_{v})$. The
pair $(Q,\alpha)$ is called a \textit{quiver setting}, and $\alpha(v)$
is referred to as the dimension of the vertex $v$. A quiver setting
is called \textit{genuine} if no vertex has dimension zero. Let $Rep_{\alpha}Q$
denote the set of representations of $Q$ with dimension vector $\alpha$,
that is: 

\[
Rep_{\alpha}Q=\bigoplus_{\rho\in A}Mat_{\alpha(t(a))\times\alpha(s(a))}(\mathbb{C}).\]

On this space we have the action of the reductive linear group

\[
GL_{\alpha}:=\bigoplus_{v\in V}GL_{\alpha(v)}(\mathbb{C}).\]

by base change. For an element $g=(g_{v_{1}},...,g_{v_{n}})$ in $GL_{\alpha}$
and a representation $X\in Rep_{\alpha}Q$ :\[
X_{\rho}^{g}=g_{t(\rho)}X_{\rho}g_{s(\rho)}^{-1}.\]

The $GL_{\alpha}$ orbits in $Rep_{\alpha}Q$ under this action are
the isomorphism classes of representations. The algebraic quotient
$Rep_{\alpha}Q//GL_{\alpha}$ will be denoted by $iss_{\alpha}Q$.
The points of $iss_{\alpha}Q$ are in bijection with the closed orbits
under the $GL_{\alpha}$ action which correspond to equivalence classes
of semisimple representations (see \cite{LeB-Proc,Lectures}). The
coordinate ring $\mathbb{C}[iss_{\alpha}Q]$ of this variety is the
ring of invariant polynomials under the $GL_{\alpha}$ action. One
can assign an invariant polynomial to a cycle $c=(\rho_{1},...,\rho_{m})$
of the quiver by defining:\[
f_{c}:Rep_{\alpha}Q\rightarrow\mathbb{C}\quad X\rightarrow Tr(X_{\rho_{1}}...\, X{}_{\rho_{n}}).\]

A cycle is called \textit{primitive} if it does not run through any
vertex more then once and \textit{quasi-primitive} for a dimension
vector $\alpha$ if the vertices that are run through more than once
have dimension bigger than 1. If $c$ is not quasi-primitive then
for some cyclic permutation of its arrows $X_{\rho_{1}}...X_{\rho_{n}}$
will be a product of 1x1 matrices and $Tr(X_{\rho_{1}}...X_{\rho_{n}})$
will be the product of the traces of these matrices, so $f_{c}$ can
be written as a product of polynomials corresponding to quasi-primitive
cycles. We recall a result of LeBruyn and Procesi \cite{LeB-Proc}:
\begin{thm}
\label{thm:PrimcGenerate}$\mathbb{C}[iss_{\alpha}Q]$ is generated
by all $f_{c}$ where $c$ is a quasi-primitive cycle with length
smaller than $\sum_{i}\alpha_{i}^{2}+1$. We can turn $\mathbb{C}[iss_{\alpha}Q]$
into a graded ring by giving $f_{c}$ the length of its cycle as degree.
\end{thm}
A quiver setting is called \textit{coregular }(resp.\textit{ a complete
intersection}) if $iss_{\alpha}Q$ is smooth (resp. complete intersection).
Noteworthily if $iss_{\alpha}Q$ is smooth then it is an affine space.
Bocklandt studied quiver settings with these two properties in \cite{Coregular}
and \cite{CI}. His method is based on establishing some reduction
steps that decrease the number of arrows and vertices in a quiver
setting $(Q,\alpha$), so that the ring of invariants of the new quiver
will only differ from $\mathbb{C}[iss_{\alpha}Q]$ in some free variables.
We recall these reduction steps. In the lemmas below $\epsilon_{v}$
denotes the dimension vector that is 1 in $v$ and 0 elsewhere. $\chi_{Q}$
denotes the Ringel form of a quiver which is defined as: \[
\chi_{Q}(\alpha,\mbox{\ensuremath{\beta})}=\sum_{v\in V}\alpha(v)\beta(v)-\sum_{\rho\in A}\alpha(s(\rho))\beta(t(\rho)).\]

\begin{lem}
\label{lem:RI}(Reduction RI: removing vertices) Suppose $(Q,\alpha)$
is a quiver setting and $v$ is a vertex without loops such that \[
\chi_{Q}(\alpha,\epsilon_{v})\geq0\quad\mbox{or}\quad\chi_{Q}(\epsilon_{v},\alpha)\geq0.\]
Construct a new quiver setting $(Q',\alpha')$ by removing the vertex
$v$ and all of the arrows $(a_{1},...,a_{l})$ pointing to $v$ and
the arrows $(b_{1},...,b_{k})$ coming from $v$, and adding a new
arrow $c_{ij}$ for each pair $(a_{i},b_{j})$ such that $s'(c_{ij})=s(a_{i})$
and $t'(c_{ij})=t(b_{j})$, as illustrated below:

\includegraphics[scale=0.7]{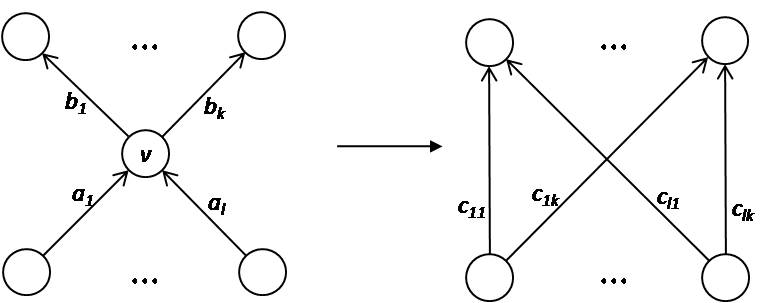} 

These two quiver settings now have isomorphic ring of invariants.
\end{lem}

\begin{lem}
\label{lem:RII}(Reduction RII: removing loops of dimension 1) Suppose
that $(Q,\lyxmathsym{\textgreek{a}})$ is a quiver setting and $v$
a vertex with $k$ loops and $\alpha(v)=1$. Take $Q'$ the corresponding
quiver without the loops of $v$, then the following identity holds:\textup{\[
\mathbb{C}[iss_{\alpha}Q]=\mathbb{C}[iss_{\alpha}Q']\otimes\mathbb{C}[X_{1},...,X_{k}],\]
} where $X_{i}$ are the polynomials that correspond to the loops
of $v$.
\end{lem}

\begin{lem}
(Reduction RIII: removing a loop of higher dimension). Suppose $(Q,\lyxmathsym{\textgreek{a}})$
is a quiver setting and $v$ is a vertex of dimension $k\geq2$ with
one loop such that \[
\chi_{Q}(\epsilon_{v},\alpha)=-1\quad\mbox{or}\quad\chi_{Q}(\alpha,\epsilon_{v})=-1.\]
(In other words, aside of the loop, there is either a single arrow
leaving $v$ and it points to a vertex with dimension 1, or there
is a single arrow pointing to $v$ and it comes from a vertex with
dimension 1). Construct a new quiver setting $(Q',\alpha')$ by changing
$(Q,\lyxmathsym{\textgreek{a}})$:

\includegraphics[scale=0.7]{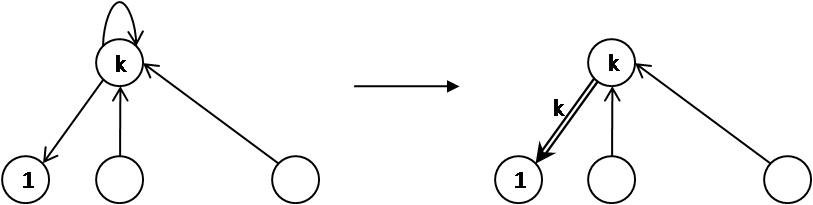}

\includegraphics[scale=0.7]{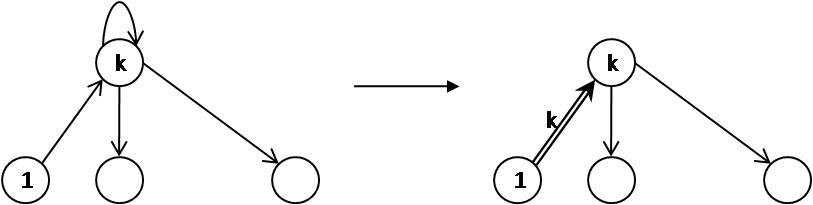}\\
We have the following identity:\textup{\[
\mathbb{C}[iss_{\alpha}Q]=\mathbb{C}[iss_{\alpha}Q']\otimes\mathbb{C}[X_{1},...,X_{k}].\]
}
\end{lem}
Quiver settings on which none of the above reduction steps can be
applied are called \textit{reduced.} Next we recall \cite[Lemma 2.4]{Coregular}
and \cite[Lemma 3.2]{CI} that show us that it is satisfactory to
consider quivers with certain connectedness properties. 
\begin{lem}
\label{lem:sc}If $(Q,\alpha)$ is a quiver setting then \[
\mathbb{C}[iss_{\alpha}Q]=\bigotimes_{i}\mathbb{C}[iss_{\alpha_{i}}Q_{i}],\]
where $Q_{i}$ are the strongly connected components of $Q$ and $\alpha_{i}=\alpha|_{Q_{i}}$.
\end{lem}
If a quiver can be decomposed into subquivers that have no arrows
running in between them and only intersect each other in vertices
of dimension one, then it is easy to see that every quasi-primitive
cycle has to run inside one of these subquivers. This inspires the
following definition:
\begin{defn}
A quiver $Q=(V,A,s,t)$ is said to be the \textit{connected sum} of
2 subquivers $Q_{1}=(V_{1},A_{1},s_{1},t_{1})$ and $Q_{2}=(V_{2},A_{2},s_{2},t_{2})$
at the vertex $v$, if the two subquivers make up the whole quiver
and only intersect in the vertex $v$. So in symbols $V=V_{1}\cup V_{2}$,
$A=A_{1}\cup A_{2}$, $V_{1}\cap V_{2}=v$, and $A_{1}\cap A_{2}=\{\emptyset\}$.
We note this by $Q=Q_{1}\#_{v}Q_{2}$. The connected sum of three
or more quivers can be defined similarly, for sake of simplicity we
will write $Q_{1}\#_{v}Q_{2}\#_{w}Q_{3}$ instead of $(Q_{1}\#_{v}Q_{2})\#_{w}Q_{3}$.
A quiver setting is called \textit{prime} if it can not be written
as a non-trivial connected sum in vertices of dimension one.
\end{defn}
Since the ring of invariants is generated by the polynomials associated
to quasi-primitive cycles, a similar result to the above can be said
about connected sums in vertices of dimension one:
\begin{lem}
\label{lem:connectedsum}Suppose $Q=Q_{1}\#_{v}Q_{2}$ $\alpha(v)=1$,
then \[
\mathbb{C}[iss_{\alpha}Q]=\mathbb{C}[iss_{\alpha_{1}}Q_{1}]\otimes\mathbb{C}[iss_{\alpha_{2}}Q_{2}],\]

where $\alpha_{1}=\alpha|_{Q_{1}}$, $\alpha_{2}=\alpha|_{Q_{2}}$.
\end{lem}
We can conclude that it is satisfactory to classify coregular or C.I.
quiver settings that are prime and strongly connected. The main result
of \cite{Coregular} is the following theorem:
\begin{thm}
\label{thm:coregular}Let $(Q,\alpha)$ be a genuine strongly connected
reduced quiver setting. Then $(Q,\alpha)$ is coregular if and only
if it is one of the three quiver settings below:

\includegraphics[scale=0.7]{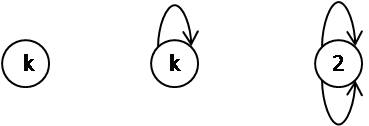}
\end{thm}

In \cite{CI} a similar result is formulated, giving a list of the
symmetric prime reduced quiver settings without loops that are C.I..
The author however makes a remark that finding such a list in the
general case is hopeless, and indeed we found that it would be a very
difficult task even in the $\alpha\equiv1$ special case which we
will investigate below. 

The proof of the above theorem uses two techniques to show that a
quiver with a complicated structure does not have the coregular or
the C.I. properties, which will also play an important role in obtaining
our results. 
\begin{defn}
A quiver $Q'=(V',A',s',t')$ is a \textit{subquiver} of the quiver
$Q=(V,A,s,t)$ if (up to graph isomorphism) $V'\subseteq V$ , $A'\subseteq A$
, $s'=s|_{A'}$ and $t'=t|_{A'}$.
\end{defn}
If $Q'$ is a subquiver of $Q$ and $\alpha'=\alpha|_{Q'}$, then
if $Q$ is coregular then $Q'$ is coregular and if $Q$ is a C.I.
then $Q'$ is a C.I., so to show that a quiver is not coregular (resp.
C.I.) it is satisfactory to show a subquiver that is not coregular
(resp. C.I.). The first statement can be found in \cite[Lemma 2.3]{Coregular},
and the second one in \cite[Lemma 4.3]{CI}.

Both smoothness and being a C.I. are properties that can be interpreted
locally. Smoothness of a variety by definition means that it is smooth
in every point. It is also known that a homogeneous ideal in a polynomial
ring is a complete intersection if its localization by the ideal of
positively graded elements is a complete intersection. Since the ideal
of relations for a quiver is a homogeneous ring (if we give the generators
the grade equal to the length of the corresponding cycles), to see
that a quiver is a C.I. it suffices to see that its localization around
the zero representation is a C.I. 

A theorem of Le Bruyn and Procesi \cite[Theorem 5]{LeB-Proc}, based
on the Luna-Slice Theorem \cite{Luna}, showed that for every point
$p\in iss_{\alpha}Q$ corresponding to a semi-simple representation,
we can build a quiver setting $(Q_{p},\alpha_{p})$ which usually
has a simpler structure, such that a neighborhood of the zero representation
in $iss_{\alpha_{p}}Q_{p}$ will be étale isomorphic to an open neighborhood
of $p$ in $iss_{\alpha}Q$. 
\begin{thm}
\label{thm:local}For a point $p\in iss_{\alpha}Q$ corresponding
to a semisimple representation $V=S_{1}^{\oplus a_{1}}\oplus...\oplus S_{k}^{\oplus a_{k}}$,
there is a quiver setting $(Q_{p},\alpha_{p})$ called the \textup{local
quiver setting} such that we have an étale isomorphism between an
open neighborhood of the zero representation in $iss_{\alpha_{p}}Q_{p}$
and an open neighborhood of $p$ in $iss_{\alpha}Q$. 

$Q_{p}$ has $k$ vertices corresponding to the set $\{S_{i}\}$ of
simple factors of $V$ and between $S_{i}$ and $S_{j}$ the number
of arrows equals \[
\delta_{ij}-\chi_{Q}(\alpha_{i},\alpha_{j}),\]
where $\alpha_{i}$ is the dimension vector of the simple component
$S_{i}$ and $\chi_{Q}$ is the Ringel form of the quiver $Q$. The
dimension vector $\alpha_{p}$ is defined to be $(a_{1},...,a_{k})$,
where the $a_{i}$ are the multiplicities of the simple components
in $V$.\end{thm}
\begin{rem}
\label{rem:LocalQ} If there is an étale isomorphism from an open
neighborhood of $p\in Q$ to an open neighborhood of $p'\in Q'$ then
$Q$ will be smooth (resp. locally C.I.) in $p$ if and only if $Q'$
is smooth (resp. locally C.I.) in $p'$. (For the latter see \cite{Greco CI}.)
Consequently to show that a quiver setting is not coregular (resp.
C.I.) it is satisfactory to find a local quiver setting that is not
coregular (resp. C.I.).
\end{rem}
The structure of the local quiver setting only depends on the dimension
vectors of the simple components. So to find all local quivers of
a given quiver setting we have to decompose $\alpha$ into a linear
combination of dimension vectors $\alpha=\sum a_{i}*\beta_{i}$ ($a_{i}\in\mathbb{N}$
and the $\beta_{i}$-s are not necessarily different) and check if
there is a semi-simple representation corresponding to this decomposition.
This depends on two conditions: there has to be a simple representation
corresponding to each $\beta_{i}$, and if some of the $\beta_{i}$-s
are the same there has to be at least as many different simple representation
classes with dimension vector $\beta_{i}$. For this we recall a theorem
from \cite{LeB-Proc}
\begin{thm}
\label{thm:Simple}Let $(Q,\alpha)$ be a genuine quiver setting.
There exist simple representations of dimension vector $\alpha$ if
and only if

- $Q$ is of the form

\includegraphics[scale=0.7]{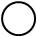}, \includegraphics[scale=0.7]{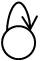},
or \includegraphics[scale=0.7]{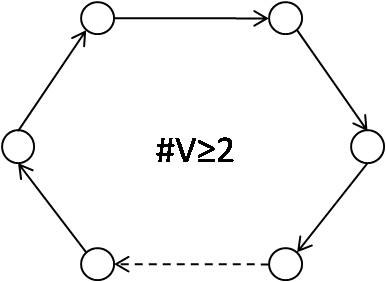}

and $\alpha(v)=1$ for all $v\in V$.

- $Q$ is not of the form above, but strongly connected and\[
\forall v\in V:\quad\chi_{Q}(\epsilon_{v},\alpha)\leq0\quad and\quad\chi_{Q}(\alpha,\epsilon_{v})\leq0\]
 where $\epsilon_{v}(u)=\begin{cases}
1 & if\: v=u\\
0 & otherwise.\end{cases}$

If $(Q,\alpha)$ is not genuine, the simple representations classes
are in bijective correspondence to the simple representations classes
of the genuine quiver setting obtained by deleting all vertices with
dimension zero. In all of the cases the dimension of $iss_{\alpha}Q$
is given by $1-\chi_{Q}(\alpha,\alpha)$, which is bigger than zero
except for the one vertex without loops, so in all the other cases
there are infinitely many classes of semi-simple representations,
and in the case of one vertex without loops there is a unique simple
representation.\end{thm}
\begin{defn}
We will call a quiver setting $(Q',\alpha')$ a descendant of a quiver
setting $(Q,\alpha)$, if $(Q',\alpha')$ can be obtained from $(Q,\alpha)$
by repeteadly applying the reduction steps RI-III, taking local quivers
or subquivers. 

So by the above, to see that quiver setting is not a C.I. it is enough
to show that one of its descendants is not a C.I. Being a descendant
is somewhat similar to the notion of minor for undirected graphs,
however in this case only certain kinds of edge-contractions are allowed.
\end{defn}

\section{\label{sub:1-dimQuivers}Quiver settings with one dimensional vertices}

In sections 3-6 we will only be dealing with the case when the dimension
vector of the quiver setting is $(1,...,1)$, and for the sake of
simplicity we will write $Q$ instead of $(Q,\alpha)$. We will briefly
overview what the above results mean when all the vertices have dimension
1. For a quiver $Q=(V,A,s,t)$ and $\alpha=(1,...,1)$, $1-\chi_{Q}(\epsilon_{v},\alpha)$
and $1-\chi_{Q}(\alpha,\epsilon_{v})$ are the in-degree and the out-degree
of the vertex $v$, and $\chi_{Q}(\alpha,\alpha)=|V|-|A|$. According
to Theorem \ref{thm:Simple} there is a simple representation with
dimension vector $\alpha$ if and only if $Q$ is strongly connected
($\chi_{Q}(\epsilon_{v},\alpha)\leq0$ and $\chi_{Q}(\alpha,\epsilon_{v})\leq0$
holds automatically in this case). Applying Theorem \ref{thm:local}
we can see that to construct a local quiver $(Q',\alpha')$ we have
to decompose $Q$ to strongly connected complete subquivers, then
the vertices of $Q'$ will correspond to these subquivers, and the
number of arrows between two vertices will equal to the number of
arrows between the corresponding subquivers of $Q.$ Since each simple
component is listed once in the decomposition, we have $\alpha'=(1,...,1).$
We will say that $Q'$ is the local quiver we get by \textit{gluing}
the vertices in some strongly connected subquivers. Note that when
we glue together some vertices in a quiver there will be a natural
graph-homomorphism between the old and the new quiver, so it makes
sense to talk about the image and pre-image of vertices, arrows, paths
and cycles. The image of a path remains a path if it did not run through
the glued subquiver twice, and it will become a cycle if it started
from and ended in the glued subquiver. The image of a cycle will always
be a cycle, but the image of a primitive cycle will only remain primitive
if it did not run through the glued subquiver twice. Moreover each
path (resp. primitive cycle) has amongst its preimages at least one
path (resp. primitive cycle).

Also for a strongly connected $Q$, $dim(iss_{\alpha}Q)=1-\mbox{\ensuremath{\chi}}_{Q}(\alpha,\alpha)=1+|A|-|V|$.
The quasi-primitive cycles and the primitive cycles are the same,
and they generate the ring of invariants. It is also clear that all
of these cycles are needed to generate that ring. Let $C$ denote
the set of primitive cycles in $Q$, $iss_{\alpha}Q$ is embedded
in a $|C|$ dimensional affine space, so \[
codim(iss_{\alpha}Q)=|C|+|V|-|A|-1.\]
For an arbitrary quiver $Q$ we will use the notation $F(Q)=|C|+|V|-|A|-1$.
(It is worth noting that we now have a geometrical proof for the combinatorial
fact that $F(Q)\geq0$ for any strongly connected quiver $Q$.) For
a quiver setting in which all vertices are 1 dimensional, $iss_{\alpha}Q$
being smooth (so an affine space) is equivalent to $F(Q)=0$, and
$iss_{\alpha}Q$ being a C.I. is equivalent to the ideal of $iss_{\alpha}Q$
being generated by $F(Q)$ elements.

We also note that RIII can never be applied on a quiver with one dimensional
vertices, so being reduced in this case means, that there is no loops
in the quiver and all the vertices have in-degree and out-degree greater
than or equal to 2, or that the quiver consists of a single vertex
with no loops.

\section{Proof of theorem \ref{thm:coregular}}

There is an error in \cite{Coregular}, in the proof for \ref{thm:coregular}.
When the author discusses the case $\alpha=(1,...,1)$, he argues
on the bottom of page 312 that when there is no subquiver of form

\includegraphics[scale=0.7]{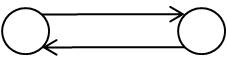}\\
then a vertex '$v$' can be removed in the following way: 

\includegraphics[scale=0.7]{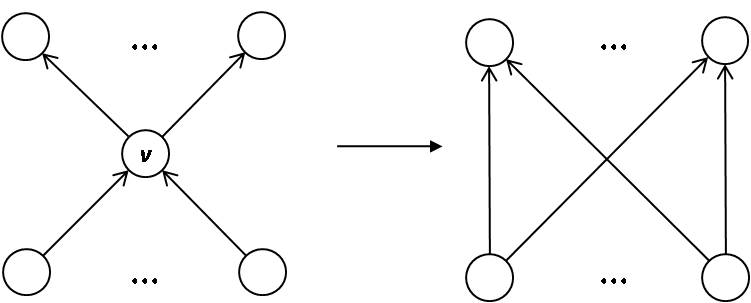}\\
without changing the number of primitive cycles. This is however
not true since non-primitive cycles, that run through $v$ multiple
times, but do not run through any other vertex more then once, will
become primitive cycles in the new quiver. The number of new cycles
can be arbitrarily large as demonstrated on the example below: 

\includegraphics[scale=0.7]{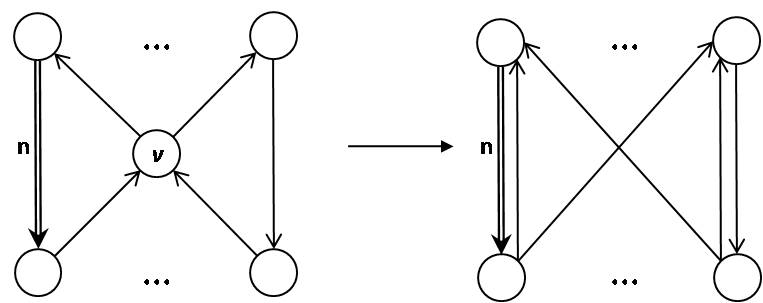}\\
The quiver on the left has $n+1$ primitive cycles, while the
one on the right has $2n+1$.

As it is explained above, to prove Theorem \ref{thm:coregular} in
the case $\alpha=(1,...,1)$, we have to see that the only reduced
quiver, for which $F(Q)=0$ holds, is the one consisting of a single
vertex with no loops. This follows from the lemma below, which will
also be useful for us later:
\begin{lem}
\label{lem:coregularcorrection} If Q is a strongly-connected quiver
without loops, and for every vertex the in-degree and the out-degree
are both at least 2, then $F(Q)\geq1$.\end{lem}
\begin{proof}
We prove the theorem by induction on the number of vertices. For one
vertex the statement is true, since there is no such quiver with one
vertex at all. Let us suppose we already saw that the statement is
true for quivers with at most $k$ vertices. It then follows that
the following stronger statement is true for quivers with at most
k vertices: 

({*}) \textit{If Q is a strongly-connected quiver, }\textbf{\textit{with
at least two vertices}}\textit{, without loops, and for every vertex,
}\textbf{\textit{with the possible exception of one vertex,}}\textit{
the in-degree and the out-degree are both at least 2, then $F(Q)\geq1$.}

We prove this by induction as well. ({*}) is obviously true if there
are only two vertices since if one of them has in-degree and out-degree
two or bigger then so does the other. Let us suppose ({*}) is true
for some $l<k$, and regard a quiver $Q$ with $l+1$ vertices that
has at most one vertex whose in- and out-degrees are not both at least
two. If it has no such vertex then $F(Q)\geq1$ follows from $k\geq l+1$
and the induction hypothesis on the original lemma. If it has exactly
one such vertex then we apply the reduction step RI, and then RII
to remove all possible loops, and get a quiver $Q'$ that has again
at most one vertex whose in- and out-degrees are not both at least
two. So applying the induction hypothesis on $Q'$ we get $F(Q')\geq1$
and since, according to lemmas \ref{lem:RI} and \ref{lem:RII} neither
RI nor RII can change this property, $F(Q)\geq1$ holds.

Now we proceed with the induction on the original lemma. Let us suppose
$Q$ is a quiver with $k+1$ vertices for which every vertex has in-
and out-degrees 2 or greater. If all primitive cycles of Q are $k+1$
long then Q has a subquiver of form:\smallskip{}

\includegraphics[scale=0.7]{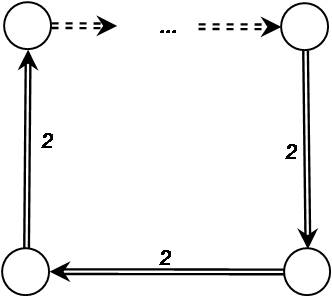}\\
for which $F(Q)\geq1$.

If there is a primitive cycle shorter than $k+1$ then let $Q'$ be
the local quiver of $Q$ we get by gluing the vertices of this cycle.
Let $Q''$ be the quiver we get from $Q'$ by removing all loops.
$Q''$ has at most one vertex that can have in- or out-degree 1 (namely
the new vertex we created by gluing the cycle), it is strongly connected
and has at least 2 vertices but no more than $k$, so we can apply
({*}) and see that $F(Q'')\geq1$. Since we got $Q''$ is a descendant
of $Q$ , $F(Q)\geq1$ follows.
\end{proof}

\section{Minimal generating sets for the ideal of relations}

According to Theorem \ref{thm:PrimcGenerate}, $\mathbb{C}[iss_{\alpha}Q]$
is isomorphic to a monomial subring of a polynomial ring. Let $n$
denote the number of primitive cycles in $Q$ and $f_{1},...,f_{n}$
denote the monomials corresponding to the primitive cycles. $iss_{\alpha}Q$
is then embedded in an $n$ dimensional affine space, and we have
a morphism \[
\varphi:\mathbb{C}[x_{1},...x_{n}]\rightarrow\mathbb{C}[iss_{\alpha}Q]\quad\varphi(x_{i})=f_{i}\]
for which $Ker(\varphi)$ is the ideal of the variety $iss_{\alpha}Q$.
We will refer to $Ker(\varphi)$ as the ideal of relations, and to
its elements as relations on $Q$. 

It is important to note that, since $\mathbb{C}[iss_{\alpha}Q]$ is
a monomial ring, the ideal $Ker(\varphi')$ (also referred to as the
\textit{toric ideal} of the monomial subring) is generated by binomials
(see for ex., \cite[Proposition 7.1.2]{Villareal}). Moreover some
binomial $m_{1}-m_{2}$ is in $Ker(\varphi)$ if and only if the multisets
of arrows corresponding to $m_{1}$ and $m_{2}$ are the same (this
is clear if we regard $\mathbb{C}[iss_{\alpha}Q]$ as a subring of
the polynomial ring $\mathbb{C}[Rep_{\alpha}Q]$).

The elements of $\mathbb{C}[iss_{\alpha}Q]$ and thus $Ker(\varphi)$
can be graded so that each variable has grade equal to the length
of the corresponding primitive cycle. With this grading all of the
binomials in $Ker(\varphi)$ are homogeneous, so $Ker(\varphi)$ is
a homogeneous ideal. It is then known that all minimal homogeneous
systems of generators of $Ker(\varphi)$ have the same number of elements,
and that a generating set with minimal number of elements can be chosen
to be homogeneous. So to find out the number of elements needed to
generate $Ker(\varphi)$ we only have to find a minimal binomial generating
set.

It is clear that one can find a minimal generating set of binomials
recursively going through the corresponding multisets of arrows (or
equivalently the multigrading inherited from $\mathbb{C}[Rep_{\alpha}Q]$).
If a minimal set $A$ generating the relations on all multisets strictly
smaller than $U$ has already been found we will have to chose a minimal
set of relations $B$ on $U$, such that $A\cup B$ generate the relations
on all multisets that are contained in $U$. We can define a relation
among the partitions of U: $m_{1}\thickapprox_{U}m_{2}$ if and only
if $m_{1}-m_{2}$ is generated by relations on multisets strictly
smaller than U . Clearly this is an equivalence relation. It is easy
to see that $m_{1}\thickapprox_{U}m_{2}$ is equivalent to saying
that there is a finite sequence $p_{1},...,p_{k}$ of partitions of
$U$ into cycles, such that $m_{1}=p_{1}$ , $m_{2}=p_{k}$ and $p_{i}$
and $p_{i+1}$ have at least one cycle in common for all $1\leq i\leq nk-1$.
If there are $n$ equivalence classes of $\thickapprox_{U}$, then
we will need at least $n-1$ relations in $B$. Thus if we can calculate
the size of a minimal generating set for $Ker(\varphi)$, by counting
the number of equivalence classes for each arrow multiset of $Q$.
Note that $\thickapprox_{U}$ only depends on the quiver and the multiset
and not how the generators in smaller multisets were chosen, so it
makes sense to use the notation \[
E(U):=|\{equivalence\, classes\, of\,\thickapprox_{U}\}|-1.\]
 Note that $E(U)=0$ with finitely many exceptions since, the ideal
of relations is finitely generated.

For two cycles $c_{1},\: c_{2}$ we will denote by $c_{1}+c_{2}$
the multiset sum of their arrows. We define an equivalence relation
amongst the primitive cycles of some arrow-multiset U as follows:
$c_{1}\sim_{U}c_{2}$ if there is a finite sequence $d_{1},...,d_{k}$
of cycles such that $c_{1}=d_{1}$, $c_{2}=d_{k}$ and $d_{i}+d_{i+1}\leq U$
for all $1\leq i\leq k-1$. 
\begin{prop}
$c_{1}\sim_{U}c_{2}$ if and only if for any $m_{1}$ containing $c_{1}$
and $m_{2}$ containing $c_{2}$ $m_{1}\thickapprox_{U}m_{2}$ . In
particular $\sim_{U}$ and $\thickapprox_{U}$ have the same number
of equivalence classes.\end{prop}
\begin{proof}
If $c_{1}\in m_{1}$and $c_{2}\in m_{2}$ for some $m_{1}\thickapprox_{U}m_{2}$,
we have a sequence $m_{1}=p_{1},...,p_{k}=m_{2}$ as above. Pick $d_{i}$
to be one of the cycles that are both in $p_{i-1}$ and $p_{i}$ for
$2\leq i\leq k$ and $d_{1}=c_{1}$, $d_{k+1}=c_{2}$. Deleting possible
repetitions from $d_{1},...,d_{k+1}$ we obtain a sequence as in the
definition of $\sim_{U}$ . Conversely if $c_{1}\sim_{U}c_{2}$ then
take a sequence $d_{1},...,d_{k}$ as in the definition and pick $p_{2},...,p_{k}$
to be such that $p_{i}$ contains $d_{i-1}+d_{i}$. Choosing $p_{1}=m_{1}$and
$p_{k+1}=m_{2}$ we see that $m_{1}\thickapprox_{U}m_{2}$.
\end{proof}
We will call a primitive cycle \emph{strong} in $U$ if its complement
is a single primitive cycle, \emph{weak }otherwise. We will call a
binomial relation strong if at least one of its monomials is a product
of two primitive cycles, weak otherwise. First we will establish a
theorem showing that for finding a minimal generating set of $Ker(\varphi)$
it is enough to consider strong binomial relations. A multiset of
arrows is called \textit{Eulerian} if every vertex has the same in-
and out-degree with respect to this multiset. 
\begin{thm}
\label{thm:weakcyc}For an Eulerian multiset of arrows U all of the
cycles that are weak in U are equivalent.\end{thm}
\begin{proof}
We will prove it by induction on the number of arrows in $U$. Let
$a_{1}\neq a_{2}$be two edges of $U$, such that $t(a_{1})=s(a_{2})$,
we will denote by $U(a_{1},a_{2})$ the multiset we obtain by replacing
the arrows $a_{1}$ and $a_{2}$ with a single arrow $a_{12}$ pointing
from $s(a_{1})$ to $t(a_{2})$. If a primitive primitive cycle $c$
of $U$ contained both $a_{1}$ and $a_{2}$ then we will denote by
$c(a_{1},a_{2})$ the primitive cycle in $U(a_{1},a_{2})$ we obtain
by replacing $a_{1}$ and $a_{2}$ in $c$ by $a_{12}$, if $c$ contains
neither $a_{1}$ nor $a_{2}$ then $c(a_{1},a_{2})$ will denote the
same primitive cycle in $U(a_{1},a_{2})$. For a primitive cycle $c$
in $U(a_{1},a_{2})$ we will denote by $c'$ the cycle of $U$ we
get by replacing $a_{12}$ by $a_{1}$ and $a_{2}$ if $c$ contained
it. Note that $c$ might be a primitive cycle or a union of two primitive
cycles. We will say that a pair $a_{1},a_{2}$ is good for some weak
primitive cycle $c$ if, $t(a_{1})=s(a_{2})$, $c$ contains either
both $a_{1}$ and $a_{2}$ or neither of them, and the image of $c(a_{1},a_{2})$
is also a weak primitive cycle in $U(a_{1},a_{2})$.

Suppose now that there are two partitions of $U$ into weak primitive
cycles $c_{1},c_{2}...c_{n}$ and $d_{1},d_{2}...d_{k}$ . If we can
find a pair of arrows $a_{1},a_{2}$ that is good for some $c_{i}$
and some $d_{j}$ then we can consider $c_{i}(a_{1},a_{2})$ and $d_{j}(a_{1},a_{2})$
which are equivalent in $U(a_{1},a_{2})$ by the induction hypothesis.
This means there is a sequence of cycles $e_{1},...,e_{n}$ in $U(a_{1},a_{2})$
as in the definition above. Clearly $e_{i}'+e_{i+1}'\leq U$ so the
sequence $e_{1}',...,e_{n}'$ shows us that $c_{i}\sim d_{j}$ (if
some of the $e_{i}'$ is a union of two primitive cycles we can replace
it by any one of these primitive cycles in the sequence). Clearly
if this holds then all of the $c_{i}$are equivalent to all of the
$d_{i}$. 

Suppose now that there is no such pair of arrows. We can suppose that
all of the primitive cycles listed contain at least two arrows, since
if $U$ contains a loop then it is disjoint form any other primitive
cycle, thus all of the primitive cycles of $U$ are equivalent. A
pair of arrows $a_{1},a_{2}$ that are consecutive in some $c_{i}$
will clearly be good for all of the $c_{j}$, and they will also be
disjoint from at least one of the $d_{i}$. If $a_{1},a_{2}$ is not
good for this $d_{i}$ then the $d_{i}(a_{1},a_{2})$ has to be a
strong primitive cycle of $U(a_{1},a_{2})$. This means that $k=3$
and in $U$ the complement of $d_{i}$ is the union of two primitive
cycles which only intersect each other in the vertex $t(a_{1})=s(a_{2})$.
This implies that $s(a_{1})\neq t(a_{2})$, so all of the $c_{i}$-s
need to have a length of at least three, thus there is at least nine
pair of arrows that appear consecutively in one of the $c_{i}$. On
the other hand for a given $d_{i}$ there can be no more than two
pairs of consecutive arrows that are disjoint from $d_{i}$ but are
not good for $d_{i}$, so at least three of the pairs that appear
consecutively in one of the $c_{i}$ will also be good for one of
the $d_{i}$. \end{proof}
\begin{cor}
\label{cor:generating set}The ideal of relations is generated by
the strong relations. For a multiset of arrows $U$ that can be obtained
as the union of two cylces, let $c_{U1}d_{U1},...,c_{Uk}d_{Uk}$ denote
the monomials corresponding to all possible partitions of $U$ into
the multiset sum of two primitive cycles, and $m_{U}$ denote the
monomial corresponding to some partition of $U$ into three or more
primitive cycles if there is such. The binomials $c_{U1}d_{U1}-c_{U2}d_{U2},\,\ldots,c_{U(k-1)}d_{U(k-1)}-c_{Uk}d_{Uk},\, c_{Uk}d_{Uk}-m_{U}$
, where $U$ ranges over all multisets as above, form a minimal generating
system for the ideal of relations. 
\end{cor}

\section{Complete intersection quiver settings}

We will say that two primitive cycles intersect each other trivially
if the multiset sum of their arrows does not yield any relations (for
arrow-disjoint primitive cycles this is same as saying that their
union is a coregular quiver). It is easy to see that two such primitive
cycles are either vertex-disjoint or their intersection is a single
directed path. 
\begin{lem}
If $c$ is a primitive cycle that intersects all of the primitive
cycles of $Q$ trivially then the local quiver $Q'$ of $Q$ obtained
by gluing the vertices of $c$ is a C.I. if and only if $Q$ is a
C.I.\end{lem}
\begin{proof}
If $Q$ is C.I. then all of its local quivers are C.I.-s. Suppose
$Q$ is not a C.I.. Since $Q'$ is a local quiver of $Q$ $codim(\mathbb{C}[iss_{\alpha}Q'])\leq codim(\mathbb{C}[iss_{\alpha}Q])$.
We can assume that the edges of $c$ are already deleted from $Q'$
(they are all loops in $Q'$). Then the primitive cycles of $Q'$
are in bijection with the primitive cycles of $Q$ that are not $c$,
since the image of a primitive cycle that intersects $c$ trivially
remains primitive in $Q'$. Let $U'$ be an Eulerian multiset in $Q'$
and $U$ be the preimage of this multiset that does not contain $c$.
The number of strong partitions of $U$ equals the sum of the numbers
of strong partitions of the multisets $U+k*c$ in $Q$. Also $U'$
has a weak partition if and only if $U$ does. So by Corollary \ref{cor:generating set}
we see that $E(U')=\sum_{k\in\mathbb{N}}E(U+k*c)$. If we sum the
left side of the equation over all Eulerian multisets of $Q'$ the
sum on the right will run over all Eulerian multisets of $Q$, so
a minimal generating set for the ideal of relations has the same size
for both quivers, thus $Q'$ can not be a C.I.. 
\end{proof}
We will call a pair of vertices a \textit{connected pair} if there
are arrows both ways between them.
\begin{thm}
\label{thm: RIV}Let $(Q,\alpha)$ be a quiver with $\alpha=(1,...,1)$
, and $(v_{1},v_{2})$ a connected pair in $Q$. Let $(Q',\alpha')$
denote the local quiver of $Q$ we get by gluing the vertices $v_{1}$
and $v_{2}$. Suppose at least one of the following holds:

a) There are exactly two paths from $v_{1}$ to $v_{2}$ and exactly
two paths from $v_{2}$ to $v_{1}$.

b) There is exactly one path from $v_{1}$ to $v_{2}$.

c) There is exactly one path from $v_{2}$ to $v_{1}$.

Then we have: $(Q,\alpha)$ is a complete intersection if and only
if $(Q',\alpha')$ is a complete intersection.\end{thm}
\begin{proof}
Cases b) and c) follow from the previous lemma. Case a) is also clear
in the case when the extra paths between $v_{1}$ and $v_{2}$ intersect
each other, since then the cycle formed by the arrows between $v_{1}$and
$v_{2}$ intersect all cycles trivially. So we only have to treat
the case when the two paths between $v_{1}$ and $v_{2}$ do not intersect
each other and thus form a primitive cycle. The arrow from $v_{1}$
to $v_{2}$ will be noted by $a_{1}$ and the the arrow from $v_{2}$
to $v_{1}$ will be noted by $a_{2}$. For the sake of simplicity
we will denote by $Q'$ the quiver from which the image of $a_{1}$
and $a_{2}$ has been deleted (since these are loops after the gluing
this does not change the C.I. property). If $U$ is a multiset of
arrows in $Q$ then we will write $U'$ for its image in $Q'$. $\Gamma_{1,2}$
will denote the two additional paths from $v_{1}$ to $v_{2}$ and
from $v_{2}$ to $v_{1}$ respectively. The cycles $(a_{1},\Gamma_{2})$
and $(a_{2},$$\Gamma_{1}$) will be noted by $q_{1}$ and $q_{2}$
, the cycle $(v_{1},v_{2})$ will be denoted by $p$ and the cycle
$(\Gamma_{1},\Gamma_{2})$ will be denoted by $s$. These are all
of the cycles that go through both $v_{1}$ and $v_{2}$ in $Q$.
It is easy to see that $Q'$ will have two less cycles, two less edges
and one less vertex compared to $Q$ , thus $codim(\mathbb{C}[iss_{\alpha}Q'])=codim(\mathbb{C}[iss_{\alpha}Q])-1$.
Using our earlier notation we need to show that \[
\sum_{U\subseteq A(Q)}E(U)\leq\sum_{U\subseteq A(Q')}E(U)+1\]
, where the left hand sum runs over all the arrow multisets of $Q$
and the right hand sum runs over all the arrow multisets in $Q'$.
The left hand sum can be written as \[
\sum_{a_{1}\in U\, or\, a_{2}\in U}E(U)+\sum_{a_{1}\notin U\, and\, a_{2}\notin U}E(U)\]
and the right hand sum can be written as \[
\sum_{a_{1}\notin U\, and\, a_{2}\notin U}E(U').\]
Let us now look at $\sum_{a_{1}\in U\, or\, a_{2}\in U}E(U)$. If
$a_{1}\in U$ but $a_{2}\notin U$ then all partitions of $U$ will
be of form $q_{1}*t$ (where $t$ is a partition of $U\backslash k_{1}$),
since the only cycle containing $a_{1}$ but not $a_{2}$ is $q_{1}$.
These partitions are all equivalent since $q_{1}*t_{1}-q_{1}*t_{2}=q_{1}*(t_{1}-t_{2})$
and $t_{1}-t_{2}$ is a relation on a proper subset of $U$. So for
such a $U$ we have $E(U)=0$, and the same can be said when $a_{1}\notin U$
and $a_{2}\in U$. Let $U_{q}$ be the multiset formed by the arrows
of $q_{1}$ and $q_{2}$. A partition of $U_{q}$ is either $q_{1}*q_{2}$
or of form $p*s$. So we have $E(U_{q})=1$. Let us now suppose that
$U$ contains both $a_{1}$and $a_{2}$ but it is not $U_{q}$. A
partition of $U$ in this case is either of form $p*t$ or of form
$q_{1}*q_{2}*\varrho$ ($t$ and $\varrho$ are monomials corresponding
to partitions of the remaining arrows). All of these are partitions
into three or more cycles, so by \ref{thm:weakcyc} in this case we
also get $E(U)=0$. Thus $\sum_{e_{1}\in U\, or\, e_{2}\in U}E(U)=1$.

If $U$ is a multiset that does not contain $a_{1}$ and $a_{2}$,
then all of its partitions into strong cycles will be also partitions
of $U'$ into strong cycles moreover if $U$ has a partition into
weak cycles then $U'$ also has a parition into weak cycles . It follows
that \[
\sum_{e_{1}\notin U\, and\, e_{2}\notin U}E(U)\leq\sum_{e_{1}\notin U\, and\, e_{2}\notin U}E(U')\]
 which is what we needed to prove.
\end{proof}
We will refer to a reduction step established in the above theorem
as RIV. Now we will proceed to show that the reduction steps RI-IV,
and decomposition into prime components will reduce all strongly connected
C.I. quivers to a single vertex without loops, and also obtain a characterization
of C.I. quivers via forbidden descendants. 

We will denote the quiver setting

\includegraphics[scale=0.7]{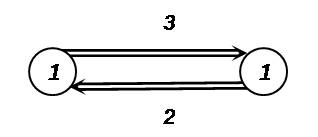}\\
by G1. and the quiver setting

\includegraphics[scale=0.7]{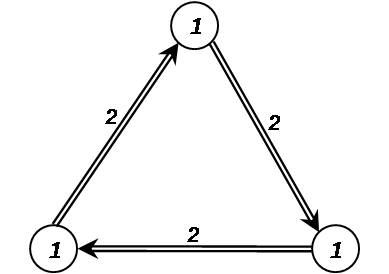}

by G2. It is easy to check that none of them are complete intersections. 
\begin{lem}
\label{lem:23conpairnotci}If $Q$ is a quiver setting in which there
is a connected pair $(v_{1},v_{2})$, and there are at least three
paths from $v_{1}$ to $v_{2}$ and at least two paths from $v_{2}$
to $v_{1}$ then G1 or G2 is a descendant of $Q$.\end{lem}
\begin{proof}
We prove by induction on the number of vertices in $Q$. If $Q$ has
two vertices, then it contains G1 as a subquiver.

We have an arrow $a_{1}$ from $v_{1}$ to $v_{2}$ and an arrow $a_{2}$
from $v_{2}$ to $v_{1}$, let $x_{1}$ and $x_{2}$ denote the paths
from $v_{1}$ to $v_{2}$ that are not $\{a_{1}\}$, and $y$ denote
the path from $v_{2}$ to $v_{1}$ that is not $\{a_{2}\}$. Let us
regard the sub-quiver $Q$' that is made of these three paths and
the arrows between $v_{1}$ and $v_{2}$. If $Q'$ has a vertex with
in- or out-degree 1, then we can apply RI (which will not change the
number of paths between $v_{1}$ and $v_{2}$) and by the induction
hypothesis $Q$' has a descendant that is G1 or G2 and so does $Q$.
So we only have to look at the cases where all vertices have in- and
out-degrees of at least 2 (so $Q$' is reduced).

First we discuss the case when $y$ consists of a single arrow from
$v_{2}$ to $v_{1}$. Let us look at the vertex where the first arrow
of $x_{1}$ points to. If it is $v_{2}$ then, since $Q$' is reduced,
it can only be G1.

If it is some other vertex $v_{3}$ then since $Q$' is reduced there
has to be another arrow pointing to $v_{3}$, which can only be part
of $x_{2}$. If we delete the arrows that are in $x_{1}$ but not
in $x_{2}$ except for the first arrow of $x_{1}$ we get a quiver
of form:\medskip{}

\includegraphics[scale=0.73]{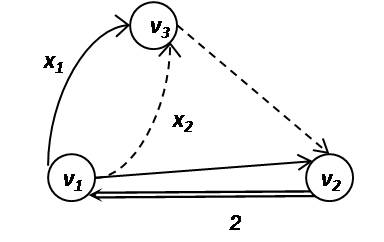}\\
(dashed arrows indicate a single directed path with arbitrarily
many vertices) which once again has a G1 descendant.

If $y$ contains more than one arrow, let $a$ denote its first arrow
and $u$ the vertex $a$ points to. Since $Q'$ is reduced $x_{1}$
or $x_{2}$ has to contain the vertex $u$ as well. We can suppose
$x_{1}$ does. The arrow $a$ and the part of $x_{1}$ that is between
$u$ and $v_{2}$ forms a directed cycle $c$. Now we can regard the
local quiver $Q''$ of $Q'$ that we get by gluing the vertices of
$c$. $Q''$ has less vertices than $Q'$. If in $Q''$ there are
still two paths from $v_{1}$ to $v_{2}$ then it satisfies the conditions
in the proposition and we are done by induction. If in $Q''$ there
is only one path both ways between $v_{1}$ and $v_{2}$ then there
was only one directed path both ways between the cycle $c$ and $v_{1}$
in $Q'$. It follows then that the segment of $x_{1}$ that is between
$v_{1}$ and $u$ is also contained in $x_{2}$. Let $Q_{0}$ be the
subquiver of $Q'$ that consists of $x_{1}$, $a$, and the part of
$x_{2}$ that starts from the first arrow in which it differs from
$x_{1}$ and ends in the first vertex which is also on $x_{1}$ (this
vertex can be $v_{2}$). $Q_{0}$ is of form:\medskip{}

\includegraphics[scale=0.7]{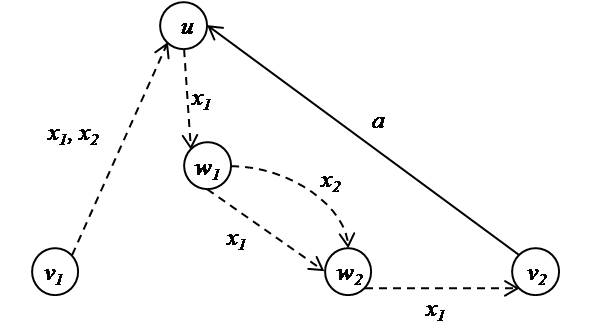}\\
(Including the special cases when $w_{1}$ is $u$ and $w_{2}$
is $v_{2}.$)

Let $Q_{1}$ be the subquiver of $Q'$ we get by adding to $Q_{0}$
the part of $y$ that is between the last vertex of $y$ that is part
of $Q_{0}$ (there is at least one such vertex: $u$) and $v_{1}$,
as well as adding the two arrows between $v_{1}$ and $v_{2}$.

Depending on where $y$ departs from $Q_{0}$, $Q_{1}$ will be one
of the following quivers:\\
\medskip{}
(I)\includegraphics[scale=0.7]{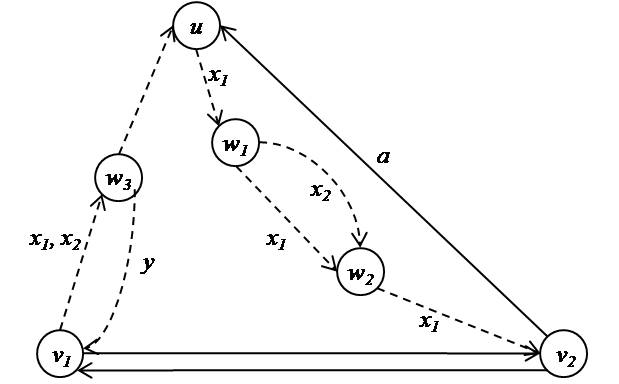}\\
\medskip{}
(II)\includegraphics[scale=0.7]{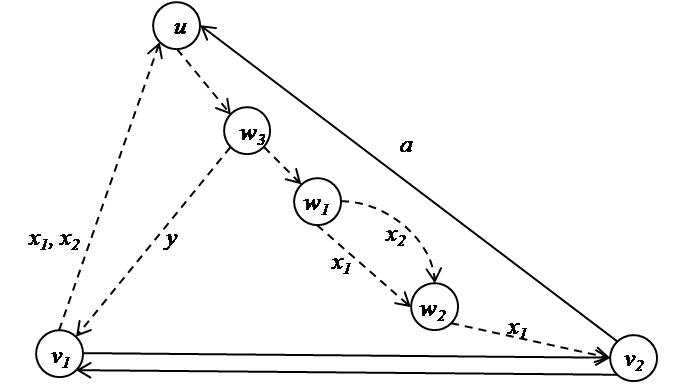}\\
\medskip{}
(III)\includegraphics[scale=0.7]{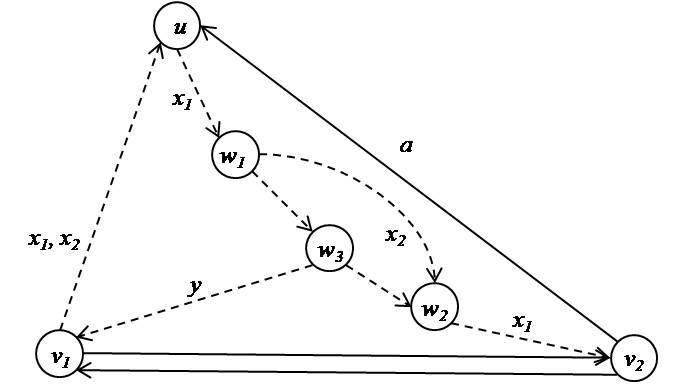}\\
\smallskip{}
(IV)\includegraphics[scale=0.7]{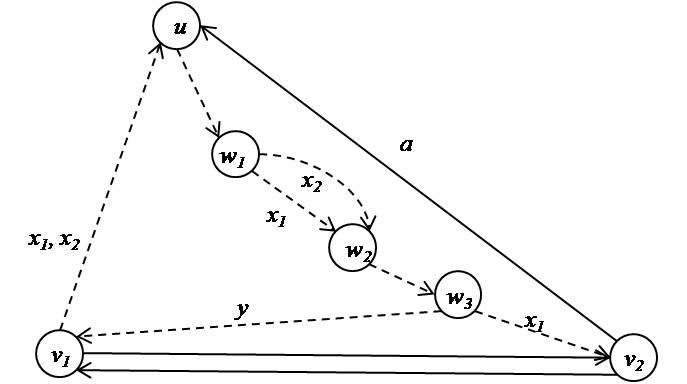}\medskip{}

Case (I): The segments of $x_{1}$ and $y$ between $v_{1}$ and $w_{3}$
form a cycle, and in the local quiver we get by gluing the vertices
of this cycle there will still be at least two paths from $v_{1}$
to $v_{2}$ and a path from $v_{2}$ to $v_{1}$ so we can apply induction.

Case (II) and (III): In both cases if we reduce by RI, we get the
following quiver:

\includegraphics[scale=0.7]{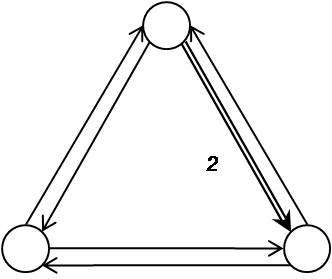}\\
The local quiver we get by gluing the bottom two vertices will
contain G1 as a subquiver.

Case(IV): If we take the local quiver we get by gluing $v_{1}$ and
$v_{2}$ , and reduce it by RI afterwards, the resulting quiver will
be G2.
\end{proof}
For technical purposes we will state the next lemma in a weaker and
stronger form and prove them simultaneously, using parallel induction.
\begin{lem}
\label{lem:noconpairs1}Let $Q$ be a reduced (by RI, RII, and RIII),
strongly connected, prime quiver setting that has at least two vertices
and does not contain any connected pairs. Then G1 or G2 is a descendant
of $Q$.
\end{lem}

\begin{lem}
.\label{lem:noconpairs2}Let $Q$ be a strongly connected, prime quiver
setting with at least three vertices and no loops. If there is a vertex
$v$ in $Q$ such that $v$ is a member of every connected pair of
$Q$ and any vertex except $v$ has in-degree and out-degree at least
2, then G1 or G2 is a descendant of $Q$..\end{lem}
\begin{proof}
We will prove by induction on the number of vertices in $Q$. Supposing
that Lemma \ref{lem:noconpairs1} holds for quivers with at most $k$
vertices and Lemma \ref{lem:noconpairs2} holds for quivers with at
most $k-1$ vertices, we will first show that Lemma \ref{lem:noconpairs2}
holds for quivers with $k$ vertices as well. Then we will use this
result to show that Lemma \ref{lem:noconpairs1} holds for quivers
with $k+1$ vertices. 

In the case $Q$ is a quiver with two vertices Lemma \ref{lem:noconpairs1}
holds trivially since a strongly connected quiver with 2 vertices
contains a connected pair. Lemma \ref{lem:noconpairs2} does not hold
for two vertices since the quiver:\\
(I)\includegraphics[scale=0.7]{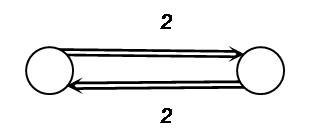}\\
satisfies the conditions. Noteworthily it is the only exception
(since one of the vertices need to have in- and out-degree 2 or greater
and any further arrows would result in the containing a G1 subquiver).

We will show the lemmas directly in the case of three vertices (it
suffices to show Lemma \ref{lem:noconpairs2}). Let the vertices be
$v$, $u_{1}$, $u_{2}$ with $u_{1}$ and $u_{2}$ having in- and
out-degrees 2 or greater, and $u_{1}$ and $u_{2}$ not forming a
connected pair. For the sake of simplicity we can suppose that there
is no arrow from $u_{2}$ to $u_{1}$. In this case there has to be
at least 2 arrows going from $u_{2}$ to $v$, and 2 arrows going
from $v$ to $u_{1}$ (because of the condition on the in- and out-degrees),
and at least one arrow going from $u_{1}$ to $u_{2}$ (because the
quiver is not prime). Moreover if there is only one arrow going from
$u_{1}$ to $u_{2}$ then there also has to be at least one arrow
going from $v$ to $u_{2}$ since the in-degree of $u_{2}$ is at
least 2. So the quiver will contain either G1 or\smallskip{}
\smallskip{}
\\
(II)\includegraphics[scale=0.7]{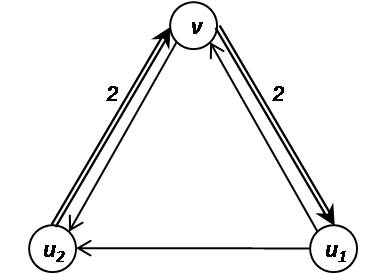}\\
 If we delete the arrow from $v$ to $u_{2}$ and apply RI on
$u_{2}$ we can see that G1 is a descendant of (II).

Now let us suppose that $Q$ has $k>3$ vertices, satisfies the conditions
in Lemma \ref{lem:noconpairs2} and is a C.I.. If there is neither
connected pairs nor a vertex with in- or out-degree less than 2 in
$Q$ we have a contradiction since we supposed that Lemma \ref{lem:noconpairs1}
holds for k vertices.

If there is a connected pair in $Q$, this pair has to contain the
vertex $v$ and some other vertex $u$. Let $Q'$ denote the local
quiver of $Q$ we get by gluing the vertices of this connected pair.
$Q'$ is strongly connected but not necessarily prime. However if
$Q'$ has non-trivial prime components $Q_{1}',\ldots,Q_{k}'$ then
it must be the connected sum $Q'=Q_{1}'\#_{v'}Q_{2}'\#_{v'}\ldots\#_{v'}Q'_{k}$
since if two prime components would meet in some vertex $w'\neq v'$
then there would be two vertices $x',y'$ in $Q'$ such that all paths
between them run through $w'$, and since $w'$ has a unique pre-image
in $Q$ the same would have to hold in $Q$ for the pre-images of
$x',y',w'$ which would contradict with $Q$ being a prime quiver.
We can conclude that all vertices in the prime components, except
$v'$, have as many in- and out-degrees as they have in $Q'$, which
is the same as their pre-images have in $Q$. Also $v'$ will be contained
in all connected pairs of $Q'$ and thus in all connected pairs of
the prime components. If $Q$ has no descendant that is G1 or G2 then
neither does $Q'$ or its prime components, and due to the induction
hypothesis on Lemma \ref{lem:noconpairs2} this means that the prime
components of $Q'$ have to be quivers with two vertices that are
of form (I) (otherwise they would satisfy the conditions in Lemma
\ref{lem:noconpairs2}). This means that every vertex in $Q$ aside
of $v$ and $u$ have exactly two arrows going to either $u$ or $v$
and exactly two arrows arriving from either $u$ or $v,$ and no other
arrows. Since $Q$ is prime, there has to be an arrow between $u$
and some other vertex $w_{1}\neq v$. We can suppose this arrow is
pointing from $u$ to $w_{1}$. Since $u$ and $w_{1}$ are not a
connected pair the two arrows leaving $w_{1}$ are both pointing to
$v$. Since $u$ has in-degree of at least 2 there has to be either
two arrows from $v$ to $u$, or an arrow pointing to $u$ from some
other vertex $w_{2}$. In the latter case there also has to be two
arrows pointing from $v$ to $w_{2}$. This means $Q$ contains one
of the following sub-quivers:\\
(V)\includegraphics[scale=0.7]{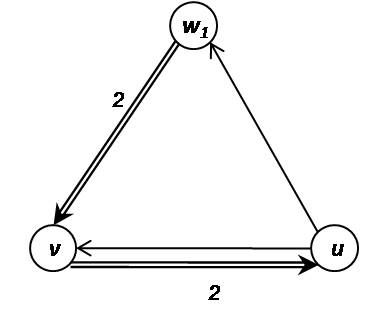}\smallskip{}
\\
(VI)\includegraphics[scale=0.7]{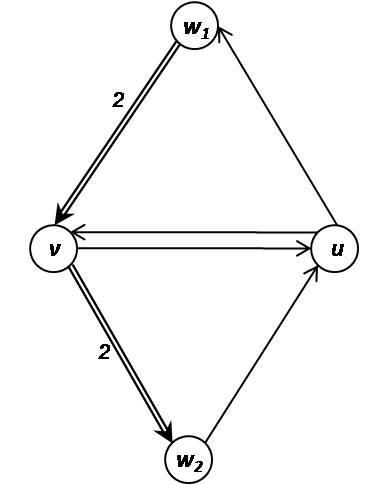}\\
both of which have a G1 descendant. (We can see this by applying
RI on $w_{1}$ in the first case and both $w_{1}$ and $w_{2}$ in
the second case.)

If $Q$ contains no connected pairs but it has a vertex $v$ with
out-degree 1, then let $u$ denote the vertex where the only arrow
leaving $v$ points to, and $Q'$ denote the quiver we get by applying
RI on $v$. Applying the same argument as above on $Q'$ we can conclude
that $Q'=Q_{1}'\#_{v'}Q_{2}'\#_{v'}\ldots\#_{v'}Q'_{k}$ where the
$Q'_{i}$-s are all of form (I). Since $u$ has in-degree of at least
2 in $Q$ there has to be an arrow entering $u$ from some vertex
$w_{2}\neq v$. Since $Q$ is strongly connected it must contain an
arrow that points to $w_{2}$, however since the prime components
of $Q'$ all meet in $v'$ this arrow can only leave from $u$ or
$v$. Since there is only one arrow leaving $v$ and that points to
$u,$ we can conclude that there is an arrow pointing from $u$ to
$w_{2}$ contradicting with the supposition that $Q$ contains no
connected pairs.

Now we are left to prove Lemma \ref{lem:noconpairs1} on $k$ vertices,
supposing that we already know that both lemmas are true for quivers
with at most $k-1$ vertices. Suppose $Q$ is a counterexample. Let
$c$ be a cycle of length $l$ in $Q$ , going through the vertices
$v_{1},v_{2},...,v_{l}$. For the sake of simplicity $v_{i}$ and
$v_{j}$ will denote the same vertex if $i\equiv j\;(mod\; l)$. If
$c$ is a cycle of minimum length in $Q$ then any arrow between the
vertices of $c$ has to point from some $v_{i}$ to $v_{i+1}$. Since
if an arrow pointed to any vertex at least two steps away in the cycle
it would form a shorter cycle than $l$ along with some of the original
arrows of $c$. Also no arrow can point from $v_{i}$ to $v_{i-1}$
since $Q$ contains no connected pairs. There can be no more than
two extra arrows going between the vertices of $c$ otherwise by deleting
everything from $Q$ but $c$ and three of these extra arrows we can
see that it has a $G2$ descendant. Also if we look at the local quiver
$Q'$ we get by gluing the vertices of this cycle, using the same
argument as above, we can see that the prime components in $Q'$ satisfy
the conditions of Lemma \ref{lem:noconpairs2} except for having at
least three vertices, so $Q'$ has to either consist of a single vertex
(if $l=k$) or be a connected sum of quivers of form (I). This means
that any vertex in Q other then $v_{1},v_{2},...,v_{l}$ will have
in- and out-degree 2 and all of its arrows will point to a vertex
in $c$ or come from a vertex in $c$. 

If the minimal cycle length in $Q$ is $k,$ then, since all vertices
have in- and out-degree of at least 2, $Q$ will contain the subquiver:\smallskip{}

\includegraphics[scale=0.7]{graf9}\\
which clearly has a G2 descendant.

If the minimal cycle length in $Q$ is $k-1$ then let $c$ be a minimal
cycle with vertices $v_{1},v_{2},...,v_{k-1}$ and $u$ be the only
vertex of $Q$ that is not in $c$. As noted above there is two arrows
going from $u$ to $c$ and two arrows going from $c$ to $u$. If
an arrow points from $v_{i}$ to $u$ and another arrow points from\linebreak{}
$u$ to $v_{j}$ then $j-i\equiv1\;(mod\; k-1)$ or
$i-j\equiv1\;(mod\; k-1)$ , otherwise the cycle $u,v_{j},v_{j+1}\ldots v_{i},u$
would be shorter than $k-1$ contradicting that $c$ is a cycle of
minimal length. For this condition to be satisfied either the two
arrows entering $u$ have to point to the same vertex or the two arrows
leaving $u$ have to come from the same vertex. We can suppose the
latter holds. Also note that if there is no arrow going from $u$
to some $v_{i}$ then there has to be at least two arrows going from
$v_{i-1}$ to $v_{i}$, and similarly if there is no arrow going from
some $v_{i}$ to $u$ then there has to be at least two arrows going
from $v_{i}$ to $v_{i+1}$ due to all vertices having in- and out-degrees
of at least 2. So depending on where the arrows departing from $u$
point to, $Q$ will contain one of the following three sub-quivers:\\
\includegraphics[scale=0.7]{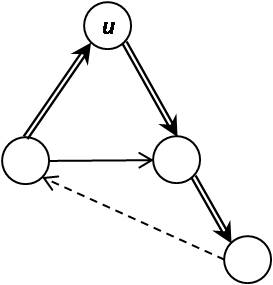}\includegraphics[scale=0.7]{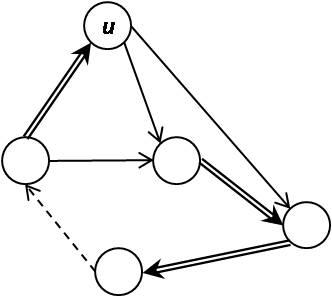}\includegraphics[scale=0.7]{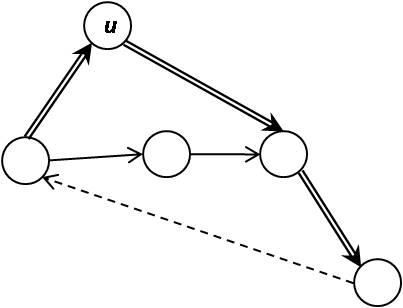}\\
All of these contain a cycle with three double vertices and thus
have a G2 descendant.

If the minimal cycle length in $Q$ is smaller than $k-1$ , there
is a minimal cycle $c$ in $Q$ with vertices $v_{1},v_{2},...,v_{l}$
and at least two vertices $u_{1}$ and $u_{2}$ outside this cycle.
Applying the same argument as above we can suppose that there is two
arrows going from $v_{1}$ to $u_{1}$. There has to be at least two
arrow pointing to $v_{1}$ and none of these can be leaving from $u_{1}$.
If both arrows entering $v_{1}$ are leaving from $v_{l}$ then we
can regard the subquiver of $Q$ consisting of the cycle $c$ the
vertex $u_{1}$ and the four arrows that are incident to $u_{1}$.
This will be of form:\medskip{}

\includegraphics[scale=0.7]{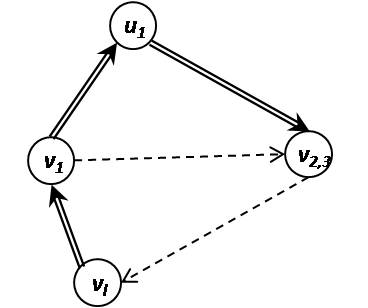}\\
or of form:

\includegraphics[scale=0.7]{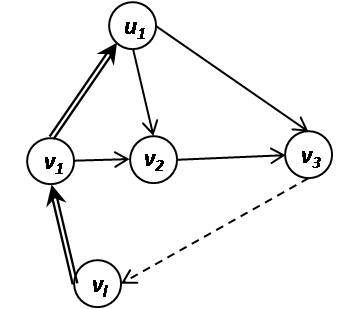}\\
The first case already contains a cycle with three double arrows,
and the second case can be reduced to the first case by applying RI
on $v_{2}$.

If one of the arrows entering $v_{1}$ are leaving from some vertex
$u_{2}$ outside the cycle $c$, let us regard the subquiver $Q'$
that consists of the cycle $c$, the vertices $u_{1},u_{2}$ and the
eight arrows that are incident to $u_{1}$ or $u_{2}.$ If there is
an arrow from $u_{1}$ to $v_{3}$ then there can be no arrow between
$u_{2}$ and $v_{2}$ (either direction) otherwise the local quiver
we get by gluing the vertices of the cycle $u_{1},v_{3},v_{4}...v_{1},u_{1}$
would not be a connected sum of quivers of form (I) (since the images
of $u_{2}$ and $v_{2}$ would be separate vertices in that quiver
with an arrow between them) contradicting the induction hypothesis
on Lemma \ref{lem:noconpairs2}. Thus if there is an arrow from $u_{1}$
to $v_{3}$ then there is only one arrow leaving $v_{2}$ in $Q'$
and thus we can apply RI, and in the resulting quiver we will have
a double arrow from $u_{1}$ to $c$ and a double arrow from $c$
to $u_{1}$. Depending on how the arrows incident to $u_{2}$ are
arranged, $Q'$ will be one of the following quivers:\\
\includegraphics[scale=0.7]{graf14}\\
or\\
\includegraphics[scale=0.7]{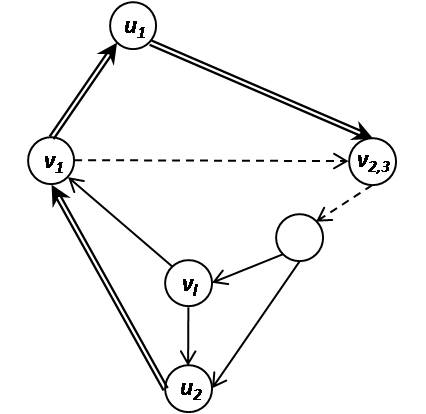}or\includegraphics[scale=0.7]{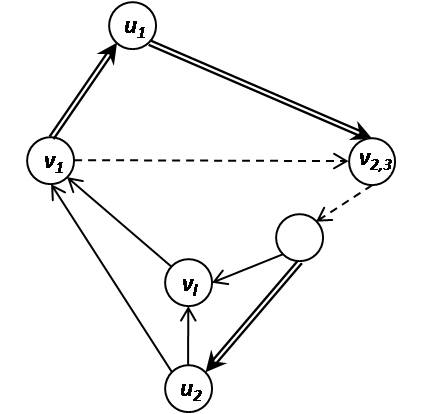}\\
all of which contain a cycle with three double arrows, therefore
have a G1 descendant.
\end{proof}
This puts us into position to state our two main results:
\begin{thm}
If $(Q,\alpha)$ is a strongly connected, prime, C.I. quiver setting
with $\alpha=(1,...,1)$ and none of the reduction steps RI-RIV can
be applied, then $Q$ consists of a single vertex with no loops.\end{thm}
\begin{proof}
Follows immediately from Lemmas \ref{lem:23conpairnotci} and \ref{lem:noconpairs1}.\end{proof}
\begin{thm}
\label{thm:Forbidden descendant}A quiver setting $(Q,\alpha)$ with
$\alpha=(1,...,1)$ is C.I. if and only if it has no G1 or G2 descendants.\end{thm}
\begin{proof}
If it is not a C.I. then we can not apply our reduction steps to it
to reduce it into a quiver with one vertex. By the two lemmas it means
it has a G1 or G2 descendant.
\end{proof}

\section{Describing coregular quivers with forbidden descendants}

Based on the results in \cite{Coregular} coregular quiver settings
(with arbitrary dimension vectors) can be characterized in a similar
manner to \ref{thm:Forbidden descendant}. The hope is that one can
obtain such result, with a relatively short list of forbidden descendants,
in the general case for C.I. quiver settings with arbitrary dimension
vector as well, which seem to be very difficult to describe via the
reduction-step method.
\begin{thm}
A quiver setting $(Q,\alpha)$ is coregular if and only if it has
no descendant of the following form:

\includegraphics[scale=0.7]{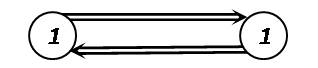}\end{thm}
\begin{proof}
We need to prove that any quiver setting except for the above that
is reduced and not coregular, has a non-trivial descendant that is
also not coregular.

Let $v$ be a vertex of $Q$ with maximal dimension. Suppose first
$\alpha(v)=1$. If $Q$ contains a non-Hamiltonian cycle consider
the strongly connected subquiver of $Q$ spanned by the vertices of
this cycle. If we take the local quiver corresponding this strongly
connected subquiver we obtain a quiver in which all vertices have
both in and out degree no less then $2$ except maybe for the new
vertex that we got gluing the vertices of the subquiver. Applying
the ({*}) version of \ref{lem:coregularcorrection} we see that this
local quiver is not coregular.

If all of the cycles of $Q$ are Hamiltionian then Q contains a subquiver
of the form 

\includegraphics[scale=0.7]{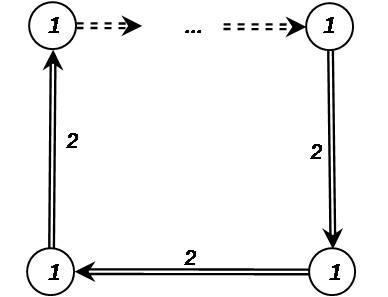}

which clearly has a descendant of the form:

\includegraphics[scale=0.7]{graf22d1}.

If $\alpha(v)\geq2$ it is easy to verify using \ref{thm:Simple}
that there is a simple representation with dimension vector $\alpha-\epsilon_{v}$
($\epsilon_{v}$ is the dimension vector that is $1$ in $\alpha$
and $0$ elsewhere). The local quiver corresponding to the decomposition
$(\alpha-\epsilon_{v})+\epsilon_{v}$ will consist of two vertices
of dimension one and the number of arrows between them will be $\chi_{Q}(\alpha-\epsilon_{v},\epsilon_{v})$
in one direction and $\chi_{Q}(\epsilon_{v},\alpha-\epsilon_{v})$
in the other. Since $(Q,\alpha)$ is reduced both of these are at
least $2$.
\end{proof}
\newpage{}

\newpage{}

\newpage{}

\end{document}